\documentclass[a4paper,11pt,twoside,reqno]{amsart}

\usepackage[utf8]{inputenc}
\usepackage[plainpages=false,pdfpagelabels=true]{hyperref}
\usepackage{amssymb,amsthm}
\usepackage[margin=1in]{geometry}
\usepackage[T1]{fontenc}
\usepackage{dsfont} 
\usepackage{xcolor} 
\usepackage[toc,page]{appendix} 
\usepackage{comment} %

\newtheorem{thm}{Theorem}[section] 
\newtheorem{proposition}[thm]{Proposition}
\newtheorem{lemma}[thm]{Lemma}
\newtheorem{corollary}[thm]{Corollary}
\theoremstyle{definition}

\newtheorem{remark}[thm]{Remark}
\newtheorem{example}[thm]{Example}

\allowdisplaybreaks[1]

\renewcommand{\epsilon}{\varepsilon}

\newcommand{\dv}{\text{ }dV}
\numberwithin{equation}{section}

\title{Construction of \texorpdfstring{$r$}{r}-harmonic submanifolds}
\author{José Miguel Balado-Alves}
\address{Universit\"at M\"unster, Mathematisches Institut\\
Einsteinstr. 62\\
48149 M\" unster\\
Germany}
\email{jose.balado@uni-muenster.de}

\author{Anna Siffert}
\address{Universit\"at M\"unster, Mathematisches Institut\\
Einsteinstr. 62\\
48149 M\" unster\\
Germany}
\email{asiffert@uni-muenster.de}

\date{\today}

\begin{document}

\begin{abstract}
    We provide a construction method for biharmonic submanifolds in cohomogeneity one manifolds. In particular, we give new examples of biharmonic submanifolds and study the normal index of these submanifolds. We use this strategy to construct metrics on the sphere admitting biharmonic non-minimal hypersurfaces with three distinct principal curvatures. Finally, we perform a similar study of $r$-harmonic submanifolds of cohomogeneity one manifolds.
\end{abstract}

\maketitle

\section{Introduction}
Throughout let $(M,g)$ and $(N,h)$ be closed Riemannian manifolds. When selecting preferred points within $\mathcal{C}^{\infty}(M,N)$, i.e. the space of smooth maps from $M$ to $N$, harmonic maps are a natural choice. They are defined to be the critical points of the energy functional $E$ which is given by
$$
E(\phi)=\frac{1}{2}\int_M |d \phi|^2 \dv_g,
$$
where $dV_g$ denotes the volume form of $M$ with respect to $g$.
Harmonic maps are characterized by the vanishing of the tension field $\tau$, where 
$$
\tau(\phi):= \operatorname{Tr}_g \nabla^{\phi} d\phi.
$$
Here $\nabla^{\phi}$ denotes the pull-back connection of the pull-back bundle $\phi^{*}TN$.
The theory of harmonic maps experienced a boost through the seminal work of Eells and Sampson~\cite{EellsSampson} from 1964 and it turned into a very rich and active area of research in subsequent years. Note that if $\phi$ is an isometric immersion, then $\phi$ is harmonic if and only if $\phi(M)$ is a minimal submanifold of $N$. 

\smallskip

Biharmonic maps are a natural generalization of harmonic maps proposed by Eells and \\ Lemaire~\cite{EellsLemaire}. They are critical points of the bienergy functional $E_2$ which is given by
$$
E_2(\phi)=\frac{1}{2}\int_M |\tau (\phi)|^2 \dv_g.
$$
Biharmonic maps are characterized by the vanishing of the  \textit{bitension field} $\tau_2$, where
$$
\tau_2(\phi):= -\Delta \tau (\phi) +\textrm{Tr}_g R^N(\tau(\phi), d\phi) d\phi.
$$
Here $\Delta$ denotes the rough Laplacian acting on sections of the pullback bundle $\phi^*TN$; it is defined by $\Delta = \operatorname{Tr}_g (\nabla_{\nabla}^{\phi} - \nabla^{\phi} \nabla^{\phi})$, where $\nabla$ denotes the connection on $M$.
Further, $R^N$ denotes the curvature tensor of $N$. Since a harmonic map is trivially biharmonic, the main interest lies in the study of non-harmonic biharmonic maps, the so-called \emph{proper biharmonic maps}. 

\smallskip

If $\phi$ is a biharmonic isometric immersion, we say that $M$ is a biharmonic submanifold of $N$. 
The study of biharmonic submanifolds was initiated by the pioneering works of Chen~\cite{Chen} and Jiang~\cite{Jiang} and is nowadays a very active research area, see e.g. the survey ~\cite{MontaldoOniciucSurvey}. 
We also refer the reader to~\cite{OuChenBook} for a more detailed discussion on the geometry of biharmonic submanifolds.

\smallskip

The existence of biharmonic maps seems to depend heavily on the curvature of the target metrics, as one can, for example, deduce from the research on the biharmonic heat flow, see e.g.~\cite{Lamm}, or the study of Chen's conjecture, see e.g. ~\cite{FuYuZhan}, which is that every biharmonic submanifold in the Euclidean space is minimal.

\smallskip

Constructing biharmonic submanifolds is, in general, a challenging task since they are solutions to a fourth-order semilinear elliptic PDE, i.e. solutions of $\tau_2(\phi)=0$. A frequently used strategy is to impose some symmetry conditions on the ambient manifold and study solutions $\phi$ of $\tau_2(\phi)=0$ which are invariant with respect to this symmetry. Classic examples appear under this ansatz: the $\operatorname{SO}(n+1)$-invariant immersion $\mathbb{S}^{n}(\frac{1}{\sqrt{2}}) \to \mathbb{S}^{n+1}$; the ${\operatorname{SO}(k+1) \times \operatorname{SO}(n-k+1) }$-invariant generalized Clifford Torus $\mathbb{S}^k(\frac{1}{\sqrt{2}}) \times \mathbb{S}^{n-k}(\frac{1}{\sqrt{2}}) \to \mathbb{S}^{n+1}$, see~\cite{Jiang}; or the biharmonic tubes over a totally geodesic $\mathbb{CP}^{n-p}$ which are invariant under the $\operatorname{SU}(p+1) \times \operatorname{SU}(n-p)$-action on $\mathbb{CP}^n$, see~\cite{Sasahara}.

\smallskip

The goal of this paper is to develop a unified framework for these (and others) examples, namely, we study biharmonic maps and biharmonic submanifolds of cohomogeneity one manifolds. 
More precisely, for $G$ a compact Lie group acting with cohomogeneity one on $(M,g)$, we study under which conditions the principal orbits of the cohomogeneity one action are biharmonic submanifolds of $M$. This approach recovers a large family of already known examples of proper biharmonic hypersurfaces, in particular the three examples mentioned above, and yields a general tool to construct biharmonic hypersurfaces of cohomogeneity one manifolds.

\smallskip

We use the developed theory to address the following problems: It has been proved in~\cite{Balmus4dim} that $(\mathbb{S}^{n+1},g_{\text{can}})$ does not admit a compact proper biharmonic hypersurface with constant mean curvature and three distinct principal curvatures. We now raise a question of a different nature: to what extent is this property characteristic for the round metric $g_{\text{can}}$? Do all metrics on $\mathbb{S}^{n+1}$ satisfy this property? We give a negative answer to this question, more precisely, we prove that there exists a metric $g_1$ in $\mathbb{S}^{n+1}$ admitting a proper biharmonic hypersurface with three distinct principal curvatures. To achieve this we make use of the tools provided before and a correct deformation of the metric preserving homogeneity of the orbits. We apply a similar strategy to construct proper biharmonic hypersurfaces in a deformation of the symmetric space $\operatorname{SU}(3)$ and $\mathbb{S}^2 \times \mathbb{S}^2$.

\smallskip

Apart from their existence, one of the most relevant questions concerning biharmonic maps is their stability. We make a systematic approach to stability for biharmonic orbits of cohomogeneity one manifolds and apply it to some of the examples constructed. In particular, we compute the normal index for some proper biharmonic hypersurfaces of the quaternionic projective space $(\mathbb{HP}^n, g_{\text{can}})$.

\smallskip

Ultimately, we show that a similar study can be carried out for higher-order energy functionals. We recover known examples and obtain new families of polyharmonic hypersurfaces in cohomogeneity one manifolds. Moreover, we construct foliations all of whose leaves are $r$-harmonic submanifolds, for each $r\in\mathbb{N}$ with $r\geq 2$.

\medskip

\noindent{\textbf{Organisation:}}
In Section\,\ref{prelim} we provide preliminaries. We construct biharmonic hypersurfaces in cohomogeneity one manifolds and
develop the theory of biharmonic maps between cohomogeneity one manifolds
whose orbit space is isometric to a closed interval in Section\,\ref{biharm}. In Section\,\ref{cheeger} we use Cheeger deformations to construct a biharmonic hypersurface with three principal curvatures in \texorpdfstring{$\mathbb{S}^7$}{S7}. In Section\,\ref{secstab} we study the stability of the biharmonic orbits constructed in Section\,\ref{biharm}.
We study the $r$-harmonic equation of a principal orbit in a cohomogeneity one manifold in Section\,\ref{rharm} and provide examples of $r$-harmonic hypersurfaces in cohomogeneity one manifolds. Section\,\ref{SecFoliations} is devoted to the construction of polyharmonic foliations.

\medskip
{\bf Acknowledgments.}
Jos\'{e} Miguel Balado-Alves gratefully acknowledges the
supports of the Deutsche Forschungsgemeinschaft (DFG, German Research Foundation) - Project-ID 427320536 - SFB 1442, as well as Germany's Excellence Strategy EXC 2044 390685587, Mathematics M\"unster: Dynamics-Geometry-Structure. Anna Siffert gratefully acknowledges the support of the Deutsche Forschungsgemeinschaft (DFG, German Research Foundation) - Project-ID 427320536 - SFB 1442. Last but not least, the authors thank Stefano Montaldo and Cezar Oniciuc very much for pointing out a missing factor in a previous version of this manuscript.

\pagebreak

\section{Preliminaries}
\label{prelim}
In the first subsection, we provide preliminaries on cohomogeneity one manifolds. The second subsection contains a brief introduction of the relevant energy functionals.

\subsection{Set up and basics on cohomogeneity one manifolds}

In this subsection, we establish notation and recall the very basics of the theory of cohomogeneity one manifolds. We refer the reader to \cite{AlexandrinoBettiol,GroveZiller,PuettmanSiffert} and the references therein. 

\smallskip

Let $(M,g)$ be a connected compact Riemannian manifold endowed with an isometric action $G\times M\rightarrow M$ of a compact Lie group $G$ on $M$ such that the orbit space $M/G$ is isometric to a closed interval $[0,L]$. We write $\pi: M \to M/G$ for the quotient map. Then the fiber $\pi^{-1}(t)$ of any interior point $t \in (0,L)$ is a principal orbit, while $\pi^{-1}(0)$ and $\pi^{-1}(L)$ are non-principal orbits.

We fix a unit-speed normal geodesic $\gamma$, i.e., a geodesic $\gamma:[0,L] \to M$ with $\gamma(0) \in \pi^{-1}(0)$ and $\gamma(L) \in \pi^{-1}(L)$ that passes through all orbits perpendicularly.
The isotropy groups $H:=G_{\gamma(t)}$ are all equal for $t \in (0,L)$.
Then every invariant smooth metric $g$ on the regular part of $M$ can be written as
\begin{equation}\label{metric}
    g=dt^2 + g_t, \qquad t \in (0,L),
\end{equation}
where $g_t$ is a smooth one-parameter family of homogeneous metrics on $G/H$. For each $t \in (0,L)$, there exists an orthogonal splitting of the tangent space $T_{\gamma(t)} M$ as the direct sum of
the tangent space $T_{\gamma(t)}(G \cdot \gamma(t))$ of the orbit $G \cdot \gamma(t)$ and the normal space to the orbit $G \cdot \gamma(t)$, which is spanned by $\dot{\gamma}(t)$.

Let $Q$ be a fixed biinvariant metric on $G$. Further, we denote the orthogonal complement of ${\operatorname{Lie}(H)=\mathfrak{h}}$ in $\mathfrak{g}$ by $\mathfrak{n}$.
Then there is a smooth $1$-parameter family of $Q$-symmetric $\operatorname{Ad}(H)$-equivariant automorphisms $P_t: \mathfrak{n} \to \mathfrak{n}$ defined by
$$
g( X^*, Y^* )_{\vert \gamma(t)}=g_t(X,Y)=Q(X, P_t Y),
$$
where for every $X\in\mathfrak{g}$, we denote by $X^*$ the associated action field, i.e.
$$
X^*_{\vert p} = \frac{d}{ds}(\operatorname{exp}sX \cdot p)_{\vert s=0}.
$$

The family $P_t$ encodes the homogeneous metrics $g_t$ and contains information about the extrinsic geometry of principal orbits. For instance, it is possible to express the shape operator $\operatorname{S}_t$ of the orbit $G \cdot \gamma(t)$ in terms of these endomorphisms:
\begin{equation}\label{ShapeOp}
    \operatorname{S}_t X^* = -\tfrac{1}{2} (P_t^{-1} \dot{P}_t X)^*,
\end{equation}
see e.g.~\cite{PuettmanSiffert}.

\smallskip

In this manuscript. we restrict ourselves to a special class of metrics that will be particularly useful. Suppose $\mathfrak{n}=\mathfrak{n}_1 \oplus \dots \oplus \mathfrak{n}_{\ell}$ is a decomposition into mutually orthogonal $\operatorname{Ad}(H)$-invariant subspaces. Consider, along $\gamma(t)$, the metric
\begin{equation}\label{DiagonalMetricDef}
    g_t:= \sum_{i=1}^{\ell}f_i(t)^2 Q_{\vert \mathfrak{n}_i},
\end{equation}
where $f_i: (0,L) \to \mathbb{R}$ are positive functions satisfying appropriate smoothness conditions as $t \to 0$ and $t \to L$, and $Q$ is a biinvariant metric on $G$. The resulting biinvariant metric $g=dt^2+g_t$ is called a \emph{diagonal} cohomogeneity one metric.

\pagebreak

\subsection{Energy functionals}
We introduce here some concepts and notations about higher-order energy functionals used throughout the manuscript. Unless otherwise stated, we assume the map ${\phi: M \to N}$ to be smooth.

\smallskip

Recall from the introduction that the energy functional and the bienergy functional are given by
$$
E(\phi)=\frac{1}{2}\int_M |d \phi|^2 \dv_g,
$$
and
$$
E_2(\phi)=\frac{1}{2}\int_M |\tau (\phi)|^2 \dv_g,
$$
respectively. 
\smallskip

More generally, one can define higher-order energy functionals as follows. If $ r = 2 s $ and $ s \in\mathbb{N}$ with $s\geq 1 $, we define the $ r $-energy functional as
\begin{equation*}
    E_{ 2 s } ( \phi ) := \frac{ 1 }{ 2 } \int_{ M } | \Delta^{ s - 1 } \tau ( \phi ) |^2 \dv_g.
\end{equation*}
If, on the other hand, $r=2s+1$, then
\begin{equation*}
    E_{ 2 s + 1 } ( \phi ) := \frac{ 1 }{ 2 } \int_{ M } \sum_{ j = 1 }^n | \nabla _{ e_j } \Delta^{ s - 1 } \tau ( \phi ) |^2 \dv_g,
\end{equation*}
where $ \{ e_i \}_{ i = 1 }^m $ is a local orthonormal frame on $ M $. We refer the reader to~\cite{PolyharmonicPseudo, PolyharmonicContinuation, PolyharmonicSpaceForms, PolyharmonicHomogeneous} for recent developments in the topic of polyharmonic maps and polyharmonic submanifolds.

\smallskip

Critical points of the $r$-energy functional are called $r$-harmonic maps.
It was proven by Maeta \cite{maeta2012k} that a smooth map $\phi: M \to N$ is $r$-harmonic if the $r$-tension field $\tau_r$ vanishes, where the $2s$-tension field is given by
\begin{align}
\label{maeta-1}
    \tau_{2s}(\phi)\notag&=\Delta^{2s-1}\tau(\phi)-\sum_{i=1}^mR^N(\Delta^{2s-2}\tau(\phi),d\phi(e_i))d\phi(e_i)\\\notag&\hspace{0.4cm}-\sum_{i=1}^m\sum_{\ell=1}^{s-1}\big\{R^N(\nabla^{\phi}_{e_i} \Delta^{s+\ell-2}\tau(\phi),\Delta^{s-\ell-1}\tau(\phi))d\phi(e_i)\\&\hspace{1.8cm}-R^N(\Delta^{s+\ell-2}\tau(\phi),\nabla^{\phi}_{e_i} \Delta^{s-\ell-1}\tau(\phi))d\phi(e_i)\big\}
\end{align}
and the $(2s+1)$-tension field is given by
\begin{align}
\label{maeta-2}
    \tau_{2s+1}(\phi)\notag&=\Delta^{2s}\tau(\phi)-\sum_{i=1}^mR^N(\Delta^{2s-1}\tau(\phi),d\phi(e_i))d\phi(e_i)\\\notag&\hspace{0.4cm}-\sum_{i=1}^m\sum_{\ell=1}^{s-1}\big\{R^N(\nabla^{\phi}_{e_i} \Delta^{s+\ell-1}\tau(\phi),\Delta^{s-\ell-1}\tau(\phi))d\phi(e_i)\\\notag&\hspace{1.8cm}-R^N(\Delta^{s+\ell-1}\tau(\phi),\nabla^{\phi}_{e_i} \Delta^{s-\ell-1}\tau(\phi))d\phi(e_i)\big\}\\&\hspace{0.4cm}-R^N(\nabla^{\phi}_{e_i} \Delta^{s-1}\tau(\phi),\Delta^{s-1}\tau(\phi))d\phi(e_i).
\end{align}
Here $\{e_i\}_{i=1}^m$ is a local orthonormal frame of $M$. 

\pagebreak

\section{Biharmonicity and cohomogenity one manifolds}
\label{biharm}

In this section, we study biharmonic hypersurfaces of cohomogeneity one manifolds and develop the theory of biharmonic maps between cohomogeneity one manifolds whose orbit space is isometric to a closed interval.

\subsection{Biharmonic hypersurfaces}\label{SectionHypersurfaces}

Let $(M^{n+1},g)$ be a compact Riemannian manifold endowed with an isometric action $G \times M \to M$ of a compact Lie group $G$ on $M$ such that the orbit space $M/G$ is isometric to the closed interval $[0,L]$, where $L\in\mathbb{R}_{+}$. Moreover, we denote by $\gamma$ a unit speed normal geodesic such that $\gamma(0)$ is contained in one of the non-principal orbits. Fix $Q$ a biinvariant metric on $G$ and define for $t \in (0, L)$ the family of endomorphisms $P_t: \mathfrak{n} \to \mathfrak{n}$ by
$$
g( X^*, Y^* )_{\vert \gamma(t)}=Q(X, P_t Y).
$$

In the following theorem, we determine the bitension field of the isometric immersions of the orbits $G\cdot \gamma(t)$ in the cohomogeneity one manifold $M$. 

\begin{thm}\label{hypersurface}
    Let $(M^{n+1},g)$ be a cohomogeneity one manifold with $M/G=[0,L]$. Assume that $g$ is a diagonal cohomogeneity one metric. Then the bitension field of the isometric immersion $\iota: G \cdot \gamma(t) \to M$ is given by
    \begin{equation}
    \label{reduced}
        \tau_2(\iota) = -\tfrac{1}{2}\operatorname{trace} (P_t^{-1} \ddot{P}_{t}) \, \tau(\iota).
    \end{equation}
\end{thm}
\begin{proof}
    Take $E_1, \dots, E_n \in \mathfrak{n}$ such that $E_{1 \vert \gamma(t)}^*, \dots, E_{n\vert \gamma(t)}^*$ form an orthonormal basis of \\ ${T_{\gamma(t)} (G \cdot \gamma(t))}$, $t \in (0,L)$, and write $T$ for the unit normal field to the principal orbits given by \\ ${T_{\vert g\cdot \gamma(t)}=g \cdot \dot{\gamma}(t)}$. 
    
    By its very definition, the tension field of the isometric immersion $\iota$ is given by
    $$
    \tau = \sum_{i=1}^n (\nabla^M_{E_i^*}E_i^*)^{\perp} = -\tfrac{1}{2}\sum_{i=1}^n g( E_i^*, (P_t^{-1} \dot{P}_t E_i)^* )_{\vert \gamma(t)} \, T = -\tfrac{1}{2} \operatorname{trace} (P_t^{-1} \dot{P}_t) \, T,
    $$
    where we made use of equation~\eqref{ShapeOp}. Here and below we omit $\iota$. 
    Since $\operatorname{trace} (P_t^{-1} \dot{P}_t)$ depends on $t$ only, we obtain
    $$\Delta \tau = -\tfrac{1}{2} \operatorname{trace} (P_t^{-1} \dot{P}_t) \Delta T.$$

    Recall $\Delta = \operatorname{Tr}_g (\nabla_{\nabla}^{\iota} - \nabla^{\iota} \nabla^{\iota})$. In what follows we determine the tangential and normal components of $\Delta T$ with respect to the orbits separately.

    \smallskip

    We start by determining the normal part of $\Delta T$.
    Note that $\operatorname{Tr}_g \nabla^{\iota}_{\cdot} \nabla^{\iota}_{\cdot} T = \frac{1}{2} \operatorname{Tr}_g \nabla^{\iota}_{\cdot} (P_t^{-1} \dot{P}_t \, \,\cdot)^*$. Its normal component is hence given by
    $$
    g( \operatorname{Tr}_g \nabla^{\iota}_{\cdot} \nabla^{\iota}_{\cdot} T, T )_{\vert \gamma(t)} = - \tfrac{1}{4} \sum_{i=1}^n |(P_t^{-1} \dot{P}_t E_i)^*|^2_g = -\tfrac{1}{4}\operatorname{trace} (P_t^{-1} \dot{P}_t)^2.
    $$

\smallskip

    Next, we determine the tangential part of $\Delta T$.
    Let $\mathfrak{n}=\mathfrak{n}_1 \oplus \dots \oplus \mathfrak{n}_{\ell}$ be a decomposition in mutually orthogonal $\operatorname{Ad}(H)$-invariant subspaces. For $E_i \in \mathfrak{n}_k$ we can write $(\nabla^{\iota}_{E_i^*} \nabla^{\iota}_{E_i^*} T)^{\text{tan}}=\frac{f_k'}{f_k} (\nabla^{\iota}_{E_i^*}E_i^*)^{\text{tan}}$, see~\eqref{DiagonalMetricDef}. By Koszul formula we get $g( \nabla_{E_i^*} E_i^*, E_j^* )_{\vert \gamma(t)} =  g( [E_i, E_j]^*, E_i^* )_{\vert \gamma(t)}$, and thus we have
    $$
     g( [E_i, E_j]^*, E_i^* )_{\vert \gamma(t)} =  Q([E_i, E_j], P_t E_i) = - Q(E_j, [E_i, P_t E_i]).
    $$
    Since the metric is diagonal we obtain $(\nabla_{E_i^*} E_i^*)^{\operatorname{tan}}=0$. In a similar manner $\operatorname{Tr}_g \nabla^{\iota}_{\nabla_{\cdot} \cdot} T=0 $. 

    \smallskip
    
    Combining the preceding considerations
    $$
    \Delta \tau = \tfrac{1}{4} \operatorname{trace} (P_t^{-1} \dot{P}_t)^2 \, \tau.
    $$
    The claim then follows by applying the identity $$\operatorname{Tr}_g R(T, \cdot\,) \,\cdot =  \frac{1}{4} \operatorname{trace} ((P_t^{-1}\dot{P}_t)^2 - 2P_t^{-1} \ddot{P}_t) \, T,$$ see~\cite[Corollary~1.10]{GroveZiller}.
\end{proof}

\pagebreak

Using this approach, we can easily retrieve familiar examples and construct new proper biharmonic solutions.

\begin{example}\label{sphere1}
    Take the round sphere $(\mathbb{S}^{n+1}, g_{\text{can}})$ endowed with the natural $\operatorname{SO}(n+1)$-action. The orbit space $\mathbb{S}^{n+1}/\operatorname{SO}(n+1)$ is isometric to the closed interval $[0, \pi]$. Fix the normal geodesic
    $$
    \gamma(t)= \cos(t) \, \xi_1 + \sin(t) \, \xi_2, \qquad t \in [0,\pi],
    $$
    where $\{\xi_i\}_{i=1}^{n+2}$ is the standard basis of $\mathbb{R}^{n+2}$, and take the biinvariant metric $Q(X,Y)=\operatorname{trace}X^{\text{tr}}Y$, $X,Y \in \mathfrak{n}$. In this case, we can write $g_{\text{can}}=dt^2 + \sin^2 t \, Q_{\vert \mathfrak{n}}$. The endomorphisms $P_t$ are thus given by
    $$
    P_t = \sin^2 (t) \mathds{ 1 }_{ n }.
    $$
    Consequently, the equation $\operatorname{trace}P_t^{-1} \ddot{P}_{t}=0$ is equivalent to
    $$
    \cot^2 t =1.
    $$
    This is, the orbit $\operatorname{SO}(n+1) \cdot \gamma(\frac{\pi}{4})$ is proper biharmonic. This recovers the already known example $\mathbb{S}^{n}(\frac{1}{\sqrt{2}}) \to \mathbb{S}^{n+1}$.
\end{example}

\begin{example}\label{sphere2}
   We consider now $(\mathbb{S}^{n+1}, g_{\text{can}})$ endowed with the $\operatorname{SO}(k+1) \times \operatorname{SO}(n-k+1)$-action, where ${1 \leq k \leq n-1}$. Fix the normal geodesic
    $$
    \gamma(t)=\cos(t) \, \xi_1 + \sin(t) \, \xi_{k+2}, \qquad t \in [0,\tfrac{\pi}{2}],
    $$
    where $\{\xi_i\}_{i=1}^{n+2}$ is the standard basis of $\mathbb{R}^{n+2}$.
    There exists an orthonormal basis of $(\mathfrak{n},Q)$ such that for any $t \in (0, \frac{\pi}{2})$
    \begin{equation*}
        P_t = \begin{pmatrix}
            \cos^2 t \, \mathds{ 1 }_{ k } & {} \\
            {} & \sin^2 t \, \mathds{ 1 }_{ n-k }
        \end{pmatrix}.
    \end{equation*}
    Hence, from Theorem~\ref{hypersurface} we obtain that the biharmonic equation is equivalent to the identity
    $$
    k\tan^2 t + (n-k) \cot^2 t = n.
    $$
    We recover in this case the solutions $\mathbb{S}^k(\frac{1}{\sqrt{2}}) \times \mathbb{S}^{n-k}(\frac{1}{\sqrt{2}}) \to \mathbb{S}^{n+1}$, which were first obtained in~\cite{Jiang}. Note that this biharmonic hypersurface is non-minimal if and only if $n \neq 2k$.
\end{example}

\begin{example}\label{CPn}
    Consider the complex projective space $\mathbb{CP}^n$ equipped with the Fubini-Study metric and the natural $\operatorname{SU}(p+1) \times \operatorname{SU}(n-p)$-action, $ 1\leq p \leq n-1$. Take the geodesic
    $$
    \gamma ( t ) = [ \cos(t) \, \xi_1 +  \sin(t) \, \xi_{p+2} ], \qquad t \in [0, \tfrac{\pi}{2}],
    $$
    expressed in homogeneous coordinates and the inner product $Q(X,Y)=-\frac{1}{2}\operatorname{trace}XY$ of $\mathfrak{n}$. There exists a basis on $(\mathfrak{n},Q)$ such that for every $t \in (0, \frac{\pi}{2})$ we can write
    \begin{equation*}
        P_t = \begin{pmatrix}
            \cos^2 t \, \mathds{ 1 }_{ 2 p } & {} & {} \\
            {} & \sin^2 t \, \mathds{ 1 }_{ 2 ( n - p - 1 ) } & {} \\
            {} &  {} & \tfrac{ \eta^2 }{ 4 } \sin^2 2 t
        \end{pmatrix},
    \end{equation*}
    where $\eta^2 = 2 \frac{ n - p - 1 }{ n - p } + 2 \frac{ p }{ p + 1 }$. For a more detailed discussion see, for instance, \cite{Balado, Urakawa}.
    
    \pagebreak
    Theorem~\ref{hypersurface} yields the following equation, where $x=\cos^2 t$:
    \begin{equation*}
        4(n+1)x^2-2(n+2p+3)x+2p+1=0.
    \end{equation*}
    This recovers the homogeneous biharmonic Hopf hypersurfaces in $\mathbb{CP}^n$, see~\cite{Balado2, Sasahara}. They are all tubes over a totally geodesic $\mathbb{CP}^{n-p}$ with radius $t$ given by the biharmonic equation.
\end{example}

\begin{example}\label{revolution}
    Any surface of revolution $M^2$ generated by the profile curve $(\phi(t), 0, \psi(t))$ \\ parametrized with unit speed can be seen as a cohomogeneity one metric with the action of $\mathbb{S}^1$. The metric can be written as $g=dt^2 + \phi^2 d \theta^2$. 
    Equation\,\ref{reduced} thus reads
    \begin{equation}
    \label{ode}
    (\phi'(t))^2 + \phi(t) \phi''(t)=0.
   \end{equation}
    Consequently
    $$
    \sigma(\theta)=(\phi(t_0) \cos \frac{\theta}{\phi(t_0)}, \phi (t_0) \sin \frac{\theta}{\phi(t_0)}, \psi(t_0)),
    $$ 
    where $\theta \in [0,2\pi]$, 
    is a biharmonic curve in $M$ where $t_0$ is a solution to the equation (\ref{ode}). This recovers~\cite[Proposition~3.2]{Caddeocurves}.
\end{example}

\begin{example}\label{quaternionicprojectivespace}
    Consider the quaternionic projective space $(\mathbb{HP}^n, g_{\text{can}})$ with the cohomogeneity one action of $\operatorname{Sp}(n) \times \operatorname{Sp}(1)$ (see~\cite{Iwata}). We take the representing geodesic $\gamma(t)=z(t) \cdot o$ for $t \in [0, \frac{\pi}{2}]$, where $o$ is the origin of $\mathbb{HP}^n=\operatorname{Sp}(n+1)/\operatorname{Sp}(n)\times \operatorname{Sp}(1)$ and
    $$
    z(t)=\operatorname{exp} t \begin{pmatrix}
            E_1 & {} \\
            {} & E_1
        \end{pmatrix} \in \operatorname{Sp}(n+1),
    $$
    where $E_1$ is the $(n+1) \times (n+1)$ matrix whose entries satisfy $(E_1)_{2,1}=1, (E_1)_{1,2}=-1$ and all other entries equal zero. Consider the inner product $Q(X,Y)=-\frac{1}{4}\operatorname{trace}XY$, $X,Y \in \mathfrak{n}$. We can take an orthonormal basis of $(\mathfrak{n},Q)$ such that for every $t \in (0, \frac{\pi}{2})$
    \begin{equation*}
        P_t = \begin{pmatrix}
            \sin^2 t \, \mathds{ 1 }_{ 4(n-1) } & {} \\
            {} & \sin^2 2t \, \mathds{ 1 }_{ 3 }
        \end{pmatrix},
    \end{equation*}
    see~\cite{Urakawa} for a detailed discussion on the derivation of $P_t$.
    
    Hence, applying Theorem~\ref{hypersurface}, we get that the biharmonic hypersurfaces are tubes over quaternionic projective hyperplanes of radius $t$, such that $x=\cos^2 t$ is a root of the polynomial
    \begin{equation}\label{quaternioniceq}
        P(x):=8(n+2)x^2 - 4(n+5)x+3.
    \end{equation}
    
    We obtain two solutions $t_{\pm}$ such that $\operatorname{Sp}(n)\times \operatorname{Sp}(1) \cdot \gamma(t_{\pm})$ are biharmonic hypersurfaces. Moreover, the only minimal orbit satisfies $x=\frac{3}{2(2n+1)}$, which is not a root of $P$. As a consequence, the solutions are proper.
\end{example}

\begin{example}\label{quadric}
    The complex quadric $Q_n= \operatorname{SO}(n+2)/\operatorname{SO}(2)\times \operatorname{SO}(n)$ admits a cohomogeneity one action of $\operatorname{SO}(n+1)$, see~\cite{Uchida}. Take the representing geodesic $\gamma(t)=z(t) \cdot o$ for $t \in [0, \frac{\pi}{2}]$, where $o$ is the origin of $Q_n=\operatorname{SO}(n+2)/\operatorname{SO}(2) \times \operatorname{SO}(n)$ and $z(t)$ is the one-parameter subgroup of $\operatorname{SO}(n+2)$ given by
    \begin{equation*}
        z(t) = \begin{pmatrix}
            \cos t  & 0 & -\sin t & {} \\
            0 & 1 & 0 & {} \\
            \sin t  & 0 & \cos t & {} \\
            {} &  {} &  {} & \mathds{ 1 }_{ n-1 }
        \end{pmatrix}.
    \end{equation*}
    \pagebreak
    Take the inner product $Q(X,Y)=-\frac{1}{2}\operatorname{trace}XY$ for $X,Y \in \mathfrak{n}$. In this case we can write for every $t \in (0, \frac{\pi}{2})$ (see~\cite{Urakawa})
    \begin{equation*}
        P_t = \begin{pmatrix}
            \cos^2 t  & {} & {} \\
            {} & \mathds{ 1 }_{ n-1 } & {} \\
            {} &  {} &  \sin^2  t \, \mathds{ 1 }_{ n-1 }
        \end{pmatrix}.
    \end{equation*}
    Equation~\eqref{reduced} is hence given by
    $$
    \tan^2 t + (n-1)\cot^2 t =n.
    $$
    Substituting $x=\cos^2 t$, this is equivalent to $(nx-1)(2x-1)=0$. The solution $x=\frac{1}{n}$ corresponds to a minimal hypersurface. Hence, for $n > 2$, the orbit $\operatorname{SO}(n+1) \cdot \gamma(\frac{\pi}{4})$ is a proper biharmonic hypersurface in $Q_n$. 
\end{example}

\begin{example}\label{SU(3)}
    Take now $\operatorname{SU}(3)$ endowed with the metric $\langle X,Y \rangle = \operatorname{trace} X \bar{Y}^{\text{tr}}$. Consider the cohomogeneity one action of $G=\operatorname{SU}(3)$ itself:
    $$
    \operatorname{SU}(3) \times \operatorname{SU}(3) \to \operatorname{SU}(3), \qquad (A,B) \mapsto ABA^{\text{tr}}.
    $$
    A normal geodesic $\gamma$ is given by
    \begin{equation*}
        \gamma(t) = \begin{pmatrix}
            \cos t  & -\sin t & 0 \\
            \sin t & \cos t & 0 \\
            0 &  0 &  1  
        \end{pmatrix}
    \end{equation*}
    for $t \in [0, \frac{\pi}{2}]$. It has been shown in~\cite{SiffertSU(3)} that for an appropriate basis one can write, for $t \in (0, \frac{\pi}{2})$,
    \begin{equation*}
        P_t = 4\begin{pmatrix}
            1  & {} & {} & {}\\
            {} & \cos^2 t \, \mathds{ 1 }_{ 2 } & {} & {} \\
            {} &  {} &  \sin^2  \tfrac{t}{2} \, \mathds{ 1 }_{ 2 } & {} \\
            {} & {} & {} & \cos^2 \tfrac{t}{2} \, \mathds{ 1 }_{ 2 }
        \end{pmatrix}.
    \end{equation*}
    Applying Theorem~\ref{hypersurface} we obtain that the biharmonicity of the orbits reduces to the equation
    $$
    \cos^2 2t=- \tfrac{1}{4} \sin^2 2t.
    $$
    Thus, we see that among the orbits of the $\operatorname{SU}(3)$-action on $\operatorname{SU}(3)$, none is biharmonic. In the next section, we will use a perturbation of these endomorphisms to construct a metric admitting $\operatorname{SU}(3)$-invariant biharmonic orbits.
\end{example}

\begin{example}\label{S2xS2}
    Consider $\mathbb{S}^2 \times \mathbb{S}^2$ equipped with the standard product metric. There is a cohomogeneity one $\operatorname{SO}(3)$-action on $\mathbb{S}^2 \times \mathbb{S}^2$ given by
    $$
    A \cdot (p_1, p_2):= (Ap_1, Ap_2), \quad p \in \mathbb{S}^2 \times \mathbb{S}^2 \subset \mathbb{R}^3 \oplus \mathbb{R}^3, \quad A \in \operatorname{SO}(3).
    $$
    Take as normal geodesic
    $$
    \gamma(t) = (\cos t \, e_1 + \sin t \, e_2, \cos t \, e_1 - \sin t \, e_2)
    $$
    for $t \in [0, \frac{\pi}{2}]$, where $\{e_1, e_2, e_3\}$ is the canonical basis of $\mathbb{R}^3$. Endow $\operatorname{SO}(3)$ with the standard biinvariant metric. 
    
    In this case we can write for $t \in (0, \frac{\pi}{2})$:
    \begin{equation*}
        P_t = \begin{pmatrix}
            2 \sin t  & {} & {} \\
            {} & 2 \cos t & {} \\
            {} &  {} &  2 
        \end{pmatrix}.
    \end{equation*}
    See~\cite{AlexandrinoBettiol} for a detailed discussion.
    
    Since $\operatorname{trace}P_t^{-1}\ddot{P}_t=-2$, we see that there are no proper biharmonic orbits for this action.
\end{example}

\subsection{Biharmonic self-maps of cohomogeneity one manifolds}

Let $(M^{n+1},g)$ be a compact Riemannian manifold endowed with an isometric action $G \times M \to M$ of a compact Lie group $G$ on $M$ such that the orbit space $M/G$ is isometric to the closed interval $[0,L]$. Further,  assume now that the Weyl group $W$ of the action is finite. Recall that the Weyl group $W$ is by definition the subgroup of the elements of $G$ that leave $\gamma$ invariant modulo the subgroup of elements that fix $\gamma$ pointwise. The Weyl group $W$ is a dihedral subgroup of $N(H)/H$ generated by two involutions that fix $\gamma (0)$ and $\gamma (L)$, respectively. It acts simply transitively on the regular segments of the normal geodesic. We assume that $\gamma$ is closed or, equivalently, that $W$ is finite.

We consider equivariant $(k,r)$-maps $\phi: M \to M$ given by
$$
g \cdot \gamma(t) \mapsto g \cdot \gamma(r(t))
$$
where $g\in G$ and $r: [0,L] \to \mathbb{R}$ is a smooth function with $r(0)=0$ and $r(L)=kL$. These maps have been introduced in~\cite{PuettmanSiffert}.

A result of Püttmann~\cite{Puettmann} ensures that the map $ \phi $ is smooth if $ k $ is of the form $ j | W | / 2 + 1 $ where $ j $ is, in general, an even integer and, if the isotropy group at $ \gamma ( L ) $ is a subgroup of the isotropy group at $ \gamma ( ( | W | / 2 + 1 ) L ) $, then $ j $ is also allowed to be an odd integer. 
The Brouwer degree of a $ ( k, r ) $-map is given by:
\begin{equation*}
    \deg \phi = \begin{cases}
            k & \text{if $ \operatorname{codim} N_0 $ and $ \operatorname{codim} N_1 $ are both odd,}\\
            + 1 & \text{otherwise,}
        \end{cases}
\end{equation*}
if $ j $ is even, and by
\begin{equation*}
    \deg \phi = \begin{cases}
            k & \text{if $ \operatorname{codim} N_0 $ and $ \operatorname{codim} N_1 $ are both odd,}\\
            0 & \text{if $ \operatorname{codim} N_0 $ and $ \operatorname{codim} N_1 $ are both even, $ | W | \not \in 4 \mathbb{Z} $,}\\
            -1 & \text{if $ \operatorname{codim} N_0 $ is even, $ \operatorname{codim} N_1 $ is odd, and $ | W | \not \in 8 \mathbb{Z} $,}\\
            +1 & \text{otherwise,}
        \end{cases}
\end{equation*}
if $ j $ is odd. 

\smallskip

In what follows we compute the bitension field of an arbitrary $(k,r)$-map $\phi: M \to M$. We assume that the metric $g$ is a diagonal cohomogeneity one metric. In this case, we have that 
$$\tau(\phi) = F(t) \dot{\gamma}(r(t)),$$ where $F$ is a function depending on $t$ only (see~\cite[Theorem~3.1 and Theorem~3.6]{PuettmanSiffert}) given by
$$
F(t)=\ddot{r}(t) + \tfrac{1}{2} \dot{r}(t) \operatorname{trace} P_t^{-1} \dot{P}_t - \tfrac{1}{2} \operatorname{trace} P_t^{-1} \dot{P}_{r(t)}.
$$

Recall that a smooth map $\phi:M\rightarrow N$ is biharmonic if it satisfies the identity
$$
\tau_2 (\phi) := -\Delta \tau (\phi) +\textrm{Tr}_g R(\tau(\phi), d\phi) d\phi =0,
$$
where $\Delta$ stands for the rough Laplacian acting on sections of the pullback bundle defined by $\Delta = \operatorname{Tr}_g (\nabla_{\nabla}^{\phi} - \nabla^{\phi} \nabla^{\phi}).$

To compute the bitension field, we need the normal and tangential derivatives of $\phi$. We recall here some identities from~\cite{PuettmanSiffert}. The derivative of $\phi$ in the normal direction is given by
$$
d\phi_{\vert \gamma(t)} \cdot \dot{\gamma}(t) = \frac{d}{dt} (\phi \circ \gamma(t)) = \dot{r} (t) \dot{\gamma}(r(t)).
$$

For each element $X$ of the Lie algebra $\mathfrak{g}$ of $G$ we write $X^*$ for the corresponding action field
$$
X^*_{\vert p} = \frac{d}{ds}(\operatorname{exp}sX \cdot p)_{\vert s=0}.
$$
Hence
$$
d \phi_{\vert \gamma(t)} \cdot X^*_{\vert \gamma(t)}= \frac{d}{ds} \phi (\operatorname{exp}sX \cdot \gamma(t))_{\vert s=0} = \frac{d}{ds}(\operatorname{exp} sX \cdot \gamma(r(t)))_{\vert s=0}= X^*_{\gamma(r(t))}.
$$

\pagebreak

We write $e_0 = \dot{\gamma}(t)$ and we choose $E_1, \dots, E_n \in \mathfrak{g}$ such that $e_1=E^*_{1 \vert \gamma(t)}, \dots, e_n=E^*_{n \vert \gamma(t)}$ form an orthonormal basis of $T_{\gamma(t)} (G \cdot \gamma(t))$. It is enough to compute $\tau_2 (\phi)$ along $\gamma(t)$.

\begin{lemma}\label{curvature}
    For a cohomogeneity one manifold with diagonal metric, we have
    $$
    \operatorname{Tr}_g (R(\tau(\phi), d\phi) d\phi)_{\vert \gamma(t)} = \tfrac{1}{4} \operatorname{trace} \left( P_t^{-1} \dot{P}_{r(t)} P_{r(t)}^{-1} \dot{P}_{r(t)}-2P_t^{-1} \ddot{P}_{r(t)}  \right) \tau(\phi)_{\vert \gamma(t)}.
    $$
\end{lemma}
\begin{proof}
    Recall that for a diagonal cohomogeneity one metric we have $\tau(\phi)_{\vert \gamma(t)} = F(t) e_0$, where $F$ is a function depending on $t$ alone. The result follows from the identity
    \begin{equation*}
        \begin{split}
            g( R(e_0, E_i^*) E_i^*, e_0 )_{\vert \gamma(r(t))} &= Q((-\tfrac{1}{2}\ddot{P}_{r(t)} + \tfrac{1}{4}\dot{P}_{r(t)}P_{r(t)}^{-1} \dot{P}_{r(t)}) E_i, E_i) \\
            &= \tfrac{1}{4} g( (P_t^{-1} \dot{P}_{r(t)} P_{r(t)}^{-1} \dot{P}_{r(t)}-2P_t^{-1} \ddot{P}_{r(t)})E_i^*, E_i^*)_{\vert \gamma(t)} 
        \end{split}
    \end{equation*}
    and $g( R(e_0, E_i^*) E_i^*, E_j^* )_{\vert \gamma(r(t))}=0$, which is an application of~\cite[Corollary~1.10]{GroveZiller}.
\end{proof}

\begin{lemma}\label{laplacian}
    For a cohomogeneity one manifold with diagonal metric, we have
    \begin{equation*}
        -\Delta \tau(\phi)_{\vert \gamma(t)} = \left[  \ddot{F} + \dot{F} \tfrac{1}{2} \operatorname{trace} P_t^{-1} \dot{P}_t- F\tfrac{1}{4} \operatorname{trace} \left( P_t^{-1} \dot{P}_{r(t)} P_{r(t)}^{-1} \dot{P}_{r(t)} \right) \right]\dot{\gamma}(r(t)).
    \end{equation*}
\end{lemma}
\begin{proof}
    Since $F$ is a function depending only on $t$, we have
    $$
        \nabla^{\phi}_{\nabla_{E_i^*} E_i^*} F e_0 = \dot{F} g( \nabla_{E_i^*} E_i^*, e_0 )_{\vert \gamma(t)} e_0 + F \nabla^{\phi}_{(\nabla_{E_i^*} E_i^*)^{\operatorname{tan}}} e_0.
    $$
    Recall that the shape operator can be written in terms of $P_t$ as $\operatorname{S}_t X^*=-\frac{1}{2}(P_t^{-1} \dot{P}_t X)^*$. The Koszul formula implies $g( \nabla_{E_i^*} E_i^*, E_j^* )_{\vert \gamma(t)}= -Q(E_j, [E_i, P_t E_i])=0$ since the metric is diagonal. Hence, $\operatorname{Tr}_g (\nabla^{\phi}_{\nabla} \tau(\phi))_{\vert \gamma(t)} = -\dot{F}\frac{1}{2} \operatorname{trace} (P_t^{-1} \dot{P}_t) \, \dot{\gamma}(r(t))$.

    On the other hand,
    $$
    \sum_{i=0}^{n+1} \nabla^{\phi}_{E_i^*} \nabla^{\phi}_{E_i^*} F e_0 = \ddot{F} e_0 +\tfrac{1}{2} F\sum_{i=1}^{n+1} \nabla^{\phi}_{E_i^*} (P_{r(t)}^{-1} \dot{P}_{r(t)} E_{i})^*,
    $$
    and the normal projection yields
    $$
    g( \operatorname{Tr}_g (\nabla^{\phi} \nabla^{\phi} \tau(\phi)), e_0 )_{\vert \gamma(r(t))} = \ddot{F} - \tfrac{1}{4} F \operatorname{trace} (P_{t}^{-1} \dot{P}_{r(t)}P_{r(t)}^{-1} \dot{P}_{r(t)}).
    $$
    It only remains to check that $(\sum_{i=1}^n \nabla^{\phi}_{E_i^*} (P_{r(t)}^{-1} \dot{P}_{r(t)} E_{i})^*)^{\operatorname{tan}}$ vanishes. By~\eqref{DiagonalMetricDef}, if $E_i \in \mathfrak{n}_k$, we can write
    $$
    \tfrac{1}{2}(\nabla^{\phi}_{E_i^*} (P_{r(t)}^{-1} \dot{P}_{r(t)} E_{i})^*)^{\operatorname{tan}} = \tfrac{f_k'(r(t))}{f_k(r(t))} (\nabla^{\phi}_{E_i^*}  E^*_{i})^{\operatorname{tan}}.
    $$
    Since the endomorphisms $P_t$ diagonalize simultaneously, $(\nabla^{\phi}_{E_i^*}  E^*_{i})^{\operatorname{tan}}=0$.
\end{proof}

As an application, we obtain the following result.

\pagebreak

\begin{thm}
    For a cohomogeneity one manifold with diagonal metric, the biharmonic equation of a $(k,r)$-map reduces to the following boundary value problem:
    \[ \begin{cases} 
      F= \ddot{r}(t) + \tfrac{1}{2} \dot{r}(t) \operatorname{trace} P_t^{-1} \dot{P}_t - \tfrac{1}{2} \operatorname{trace} P_t^{-1} \dot{P}_{r(t)}  \\[1ex]
      \ddot{F} + \dot{F} \tfrac{1}{2} \operatorname{trace} P_t^{-1} \dot{P}_t - F\tfrac{1}{2} \operatorname{trace} P_t^{-1} \ddot{P}_{r(t)} =0
   \end{cases}
    \]
    with
    $$
    \lim_{ t \rightarrow 0 } r ( t ) = 0 \quad \text{and} \quad \lim_{ t \rightarrow L} r ( t ) = k L.
    $$
\end{thm}

\begin{remark}
    Note that if we take the cohomogeneity one action of $\operatorname{SO}(n+1)$ in $\mathbb{S}^{n+1}$ we recover~\cite[Remark~2.4]{MontaldoModels}.
\end{remark}

\section{Biharmonic hypersurfaces with three principal curvatures in \texorpdfstring{$\mathbb{S}^7$}{S7}}
\label{cheeger}

It has been proved in~\cite{Balmus4dim} that $(\mathbb{S}^{n+1},g_{\text{can}})$ does not admit a compact proper biharmonic hypersurface with constant mean curvature and three distinct principal curvatures. In fact, to the best of our knowledge, the only known biharmonic hypersurfaces in the sphere have $1$ or $2$ distinct principal curvatures: the generalized Clifford torus and the small hypersphere, see Example~\ref{sphere1} and Example~\ref{sphere2}.

\smallskip

We now raise a question of a different nature: to what extent is this property characteristic of the round metric $g_{\text{can}}$? 
Do all metrics satisfy this property?

To answer these questions, we exploit the technique developed in the previous section. The strategy will consist of starting with an isoparametric foliation of $\mathbb{S}^7$ with $3$ different principal curvatures (which is known to be homogeneous) and deform the metric while we keep the homogeneity of the hypersurfaces. For this purpose, in the next subsection, we recall some basic facts about the so-called \lq Cheeger deformation\rq. This deformation of the metric was used by Cheeger~\cite{Cheeger} to study manifolds admitting metrics of non-negative sectional curvature. For a modern description of the method, we refer the reader to~\cite{AlexandrinoBettiol}.

Finally, we also employ this method to construct biharmonic hypersurfaces in manifolds diffeomorphic to $\operatorname{SU}(3)$ and $\mathbb{S}^2 \times \mathbb{S}^2$.

\subsection{Cheeger deformation}
Let $(M^{n+1},g)$ be a compact Riemannian manifold and $G$ a compact Lie group acting isometrically on $M$. Fix a biinvariant metric $Q$ on $G$, and endow the manifold $M \times G$ with the metric $g \oplus \frac{1}{s}Q$, where $s$ is a positive real parameter. This Riemannian manifold admits a free isometric $G$-action given by
\begin{equation}\label{action}
    h \cdot (p,g) := (hp, hg), \qquad p \in M, \, g,h \in G.
\end{equation}

The orbit space $(M \times G)/G$ is diffeomorphic to $M$. For the quotient map
$$
\rho: M \times G \to M, \qquad \rho(p,g)=g^{-1}p,
$$
there is a unique metric $g_s$ on $M$ such that $\rho: (M \times G, g \oplus \frac{1}{s}Q) \to (M,g_s)$ is a Riemannian submersion. 

The isometric $G$-action on $(M \times G, g \oplus \frac{1}{s}Q)$ given by
$$
k* (p,g) := (p,gk^{-1}), \qquad p \in M, \, g, k \in G
$$
commutes with~\eqref{action} and hence descends to an isometric $G$-action on the corresponding orbit space $(M,g_s)$, which is precisely the original $G$-action on $M$.

The $1$-parameter family of metrics $g_s$ on $M$ varies smoothly with $s$, and extends smoothly to $s=0$, with $g_0=g$. As $s \to \infty$, $(M,g_s)$ converges in the Gromov-Hausdorff sense to the orbit space $M/G$.

\pagebreak

Consider a diagonal cohomogeneity one metric $g=dt^2 + g_t$ on $M$. We can perform a Cheeger deformation on the ``homogeneous'' part, obtaining $g_s=dt^2 + g_{s,t}$. The metrics $g_s$ are also diagonal cohomogeneity one for every $s >0$. Moreover, for every $s \geq 0$, the endomorphisms $P_{s,t}: \mathfrak{n} \to \mathfrak{n}$ defined by $g_{s,t}(X^*,Y^*)=Q(P_{s,t}X,Y)$ are given by
$$
P_{s,t}=P_t (\operatorname{Id + s P_t})^{-1},
$$
where $P_t \equiv P_{0,t}$.

\subsection{Isoparametric hypersurfaces with \texorpdfstring{$g=3$}{g=3}}
In the case of $3$ different principal curvatures with multiplicity $2$, the isoparametric hypersurfaces in $\mathbb{S}^7$ are tubes of constant radius of a standard Veronese embedding of $\mathbb{CP}^2$. I.e., they are diffeomorphic to $SU(3)/T^2$, whose isotropy representation splits into three inequivalent $2$-dimensional $T^2$-modules all of which are inequivalent to the submodules of $\mathfrak{h}$. By Schur's Lemma, the $H$-equivariant endomorphism $P_t$ is diagonal for any regular time $t \notin \frac{\pi}{3} \mathbb{Z}$. 

For any $t \in (0, \frac{\pi}{3})$, we get
\begin{equation*}
        P_t = \begin{pmatrix}
            \sin^2 (t) \, \mathds{ 1 }_{ 2 } & {} & {} \\
            {} & \sin^2 (t-\tfrac{\pi}{3}) \, \mathds{ 1 }_{ 2 } & {} \\
            {} &  {} &  \sin^2  (t-2\tfrac{\pi}{3}) \, \mathds{ 1 }_{ 2 }
        \end{pmatrix}.
    \end{equation*}

Performing a Cheeger deformation, we obtain, for any $(s,t) \in \mathbb{R}_{\geq 0} \times (0, \frac{\pi}{3})$,
\begin{equation*}
        P_{s,t} = \begin{pmatrix}
            \tfrac{\sin^2 (t)}{1+s\sin^2 (t)} \, \mathds{ 1 }_{ 2 } & {} & {} \\
            {} & \tfrac{\sin^2 (t-\tfrac{\pi}{3})}{1+s\sin^2 (t-\tfrac{\pi}{3})}  \, \mathds{ 1 }_{ 2 } & {} \\
            {} &  {} &  \tfrac{\sin^2 (t-2\tfrac{\pi}{3})}{1+s\sin^2 (t-2\tfrac{\pi}{3})}  \, \mathds{ 1 }_{ 2 }
        \end{pmatrix}.
    \end{equation*}

The biharmonic equation of a homogeneous hypersurface in $(\mathbb{S}^7, g_s)$, $s\geq 0$, reduces to
$$
 \operatorname{trace} P_{s,t}^{-1} \ddot{P}_{s,t} =0,
$$
where the derivative is taken with respect to $t$. 

Since $g_0=g$, the condition $\operatorname{trace} P_{0,t}^{-1} \ddot{P}_{0,t}=0$ reads
$$
2(4x-3)^2x=-1
$$
where $x=\cos^2 t$. Consequently, there are no proper biharmonic orbits in $(M,g_0)$.

By continuity, for small values of $s$,
the identity $
 \operatorname{trace} P_{s,t}^{-1} \ddot{P}_{s,t} =0
$ admits no solutions. However, with enough deformation of the metric, new solutions appear. Take, for instance, $s=1$. In this case:
\begin{equation}\label{biharmonic3ppalcurvs}
    \left( \operatorname{trace} P_{1,t}^{-1} \ddot{P}_{1,t} \right)_{\vert t=\frac{\pi}{6}} < 0 < \left( \operatorname{trace} P_{1,t}^{-1} \ddot{P}_{1,t} \right)_{\vert t=\frac{\pi}{4}}.
\end{equation}
Hence, there exists $t_0 \in (\frac{\pi}{6},\frac{\pi}{4})$ such that $(\operatorname{trace} P_{1,t}^{-1} \ddot{P}_{1,t})_{\vert t=t_0}=0$.

\smallskip

It remains to study whether this solution corresponds to a minimal hypersurface, i.e. if the biharmonic hypersurface is proper or not. A lengthy calculation shows that the minimal equation $\operatorname{trace} P_{s,t}^{-1}\dot{P}_{s,t}=0$ reduces to
$$
(4\cos^2 t -3)(3s+4)^2\cos t =0.
$$
The solution does not depend on $s$: the only minimal orbit is the fiber of $\frac{\pi}{6}$ for every $s \geq 0$. Hence, the solution obtained in~\eqref{biharmonic3ppalcurvs} is a proper biharmonic hypersurface with three different principal curvatures in $(\mathbb{S}^7, g_1)$.

\pagebreak

\begin{remark}
        A similar strategy may be applied to other cohomogeneity one actions on spheres. For instance, if one considers a cohomogeneity one action on $\mathbb{S}^9$ with $g=4$ different principal curvatures each of multiplicity $2$, or on $\mathbb{S}^{13}$ with $g=6$ different principal curvatures each of multiplicity $2$, the endomorphisms $P_t$ diagonalize simultaneously for any regular time $t \neq \frac{\pi}{g}\mathbb{Z}$ (see~\cite[Subsection~7.1]{PuettmanSiffert}).

        We can write
        $$
        P_t=\operatorname{diag}(\sin^2 (t) \, \mathds{1}_2, \sin^2 (t-\tfrac{\pi}{4}) \, \mathds{1}_2,\sin^2 (t-\tfrac{\pi}{2}) \, \mathds{1}_2,\sin^2 (t-3\tfrac{\pi}{4}) \, \mathds{1}_2)
        $$
        in the first case and
        $$
        P_t=\operatorname{diag}(\sin^2 (t) \, \mathds{1}_2, \sin^2 (t-\tfrac{\pi}{6}) \, \mathds{1}_2,\sin^2 (t-\tfrac{\pi}{3}) \, \mathds{1}_2,\sin^2 (t-\tfrac{\pi}{2}) \, \mathds{1}_2,\sin^2 (t-2\tfrac{\pi}{3}) \, \mathds{1}_2,\sin^2 (t-5\tfrac{\pi}{6}) \, \mathds{1}_2)
        $$
        in the latter.

        With the help of a computer, one can check that for any Cheeger deformation $g_s$, the only minimal orbit is the fiber of $\frac{\pi}{2g}$. In both cases $\operatorname{trace} P_{s,t}^{-1} \ddot{P}_{s,t}=4t^{-2} + o(t^{-2})$ as $t \to 0$ and $\operatorname{trace} P_{s,\frac{\pi}{2g}}^{-1} \ddot{P}_{s,\frac{\pi}{2g}}=-Cs^{-1} + o(s^{-1})$ as $s \to \infty$ for some (possible different) positive constant $C$. We deduce then the existence of $s_0$ such that if $s>0$ then there exists a proper biharmonic hypersurface with $4,6$ different principal curvatures in $(\mathbb{S}^9,g_s), (\mathbb{S}^{13},g_s)$, respectively.
\end{remark}

\subsection{Biharmonic hypersurfaces in \texorpdfstring{$\operatorname{SU}(3)$}{SU(3)}}
We have seen in Example~\ref{SU(3)} that there are no biharmonic orbits for the action of $\operatorname{SU}(3)$ in $\operatorname{SU}(3)$ with the metric $\langle X, Y \rangle= \operatorname{trace} X\bar{Y}^{\text{tr}}$. Recall that in this case we can write for $t \in (0, \frac{\pi}{2})$
\begin{equation*}
        P_t = 4\begin{pmatrix}
            1  & {} & {} & {}\\
            {} & \cos^2 t \, \mathds{ 1 }_{ 2 } & {} & {} \\
            {} &  {} &  \sin^2  \tfrac{t}{2} \, \mathds{ 1 }_{ 2 } & {} \\
            {} & {} & {} & \cos^2 \tfrac{t}{2} \, \mathds{ 1 }_{ 2 }
        \end{pmatrix}.
    \end{equation*}
Applying a cohomogeneity one Cheeger deformation yields
\begin{equation*}
        P_{s,t} = \begin{pmatrix}
            \tfrac{4}{1+4s} & {} & {} & {} \\
            {} & \tfrac{4\cos^2 (t)}{1+4s\cos^2 (t)} \, \mathds{ 1 }_{ 2 } & {} & {} \\
            {} &  {} &  \tfrac{4\sin^2 (\frac{t}{2})}{1+4s\sin^2 (\frac{t}{2})} \, \mathds{ 1 }_{ 2 } & {} \\
            {} &  {} & {} & \tfrac{4\cos^2 (\frac{t}{2})}{1+4s\cos^2 (\frac{t}{2})} \, \mathds{ 1 }_{ 2 }
        \end{pmatrix}.
    \end{equation*}
The equation $\operatorname{trace} P_{0,t}^{-1} \ddot{P}_{0,t}=0$ has no solution. If we take again $s=1$, a computation shows
\begin{equation*}
    \left( \operatorname{trace} P_{1,t}^{-1} \ddot{P}_{1,t} \right)_{\vert t=\frac{\pi}{6}}=-1 < 0 < \lim_{t\to 0} \operatorname{trace} P_{1,t}^{-1} \ddot{P}_{1,t} = + \infty.
\end{equation*}
Therefore, there exists $t_0 \in (0, \frac{\pi}{6})$ such that $\left( \operatorname{trace} P_{1,t}^{-1} \ddot{P}_{1,t} \right)_{\vert t=t_0}=0$,
 i.e. $t_0$ corresponds to a biharmonic solution.

\pagebreak
 
Moreover, using the identities $\cos t = 2 \cos^2 \frac{t}{2}-1$ and $\sin t = 2\sin \frac{t}{2} \cos \frac{t}{2}$, equation $\operatorname{trace}P_{1,t}^{-1} \dot{P}_{1,s}=0$ is equivalent to
$$
(8 \cos^2t -3) (2 \cos^2 t + 3)=0,
$$
which has no solutions for $t \in (0, \frac{\pi}{6})$. Thus, the biharmonic hypersurface is non-minimal.

\subsection{Biharmonic hypersurfaces in \texorpdfstring{$\mathbb{S}^2 \times \mathbb{S}^2$}{S2xS2}}

For any biharmonic curve $\gamma$ in $\mathbb{S}^2$, the hypersurface $\mathbb{S}^2 \times \gamma$ is biharmonic in $\mathbb{S}^2 \times \mathbb{S}^2$ equipped with the product metric. We aim here to construct more intricate solutions.

A numerical analysis suggests that for the $\operatorname{SO}(3)$ action on $\mathbb{S}^2 \times \mathbb{S}^2$ given in Example~\ref{S2xS2} no Cheeger deformation yields proper biharmonic orbits. 

We now consider a different action. Take $\mathbb{S}^2 \times \mathbb{S}^2=\mathbb{CP}^1 \times \mathbb{CP}^1$ equipped with the usual product metric and the diagonal $\operatorname{SU}(2)$-action. Endow $\operatorname{SU}(2)$ with the standard biinvariant metric. It has been proved in~\cite{urakawa1988equivariant} that for the representing geodesic $$
\gamma(t)=([\cos t : \sin t], [\cos t : -\sin t]), \qquad t \in [0, \tfrac{\pi}{4}],
$$
one can write
\begin{equation*}
        P_t = \begin{pmatrix}
            \sin^2 2t  & {} & {} \\
            {} & 1 & {} \\
            {} &  {} &  \cos^2 2t 
        \end{pmatrix}.
\end{equation*}
A computation shows that there are no proper biharmonic orbits in this case. However, if we apply a Cheeger deformation as before, for $s=1$ we obtain
$$
\left(\operatorname{trace}P_{1,t}^{-1} \ddot{P}_{1,t}\right)_{t=\frac{\pi}{8}}<0<\lim_{t\to 0} \operatorname{trace} P_{1,t}^{-1} \ddot{P}_{1,t} = + \infty.
$$
Thus, there exist $t \in (0, \frac{\pi}{8})$ such that $\operatorname{trace}P_{1,t}^{-1} \ddot{P}_{1,t}=0$. Moreover, the only minimal orbit for $s=1$ is the fiber of $\frac{\pi}{8}$, obtaining in this way that $\operatorname{SU}(2) \cdot \gamma(t)$ is a proper biharmonic hypersurface in $(\mathbb{S}^2 \times \mathbb{S}^2, g_1)$.

\section{Stability}\label{secstab}

In this section, we study the stability of the biharmonic orbits in terms of the endomorphisms $P_t$.

\smallskip

Recall the definition of the index of biharmonic maps:
Let $\phi:M\rightarrow N$ be a biharmonic map and $\{\phi_{s,t}\}$, $-\epsilon<s,t<\epsilon$ a two-parameter smooth variation of $\phi$, in particular $\phi_{0,0}=\phi$. The associated variational vector fields $V_1,V_2$ of $\{\phi_{s,t}\}$ are given by
\begin{align*}
V_1(p)=\tfrac{\partial}{\partial t}_{\lvert t=0}\phi_{t,0}(p), \quad
V_2(p)=\tfrac{\partial}{\partial s}_{\lvert s=0}\phi_{0,s}(p).
\end{align*}
Thus $V_1, V_2$ are sections of the pullback bundle $\phi^*TN$. The Hessian
of the bienergy functional $E_2$ at a critical point $\phi$ is defined by
\begin{align*}
    H_{\phi}(V_1,V_2)=\tfrac{\partial^2}{\partial t \partial s}_{\lvert (t,s)=(0,0)}E_2(\phi_{s,t}).
\end{align*}

\begin{thm}[\cite{JiangImmersions}]
    For a biharmonic map $\phi:M\rightarrow N$ between two Riemannian manifolds, where $M$ is compact, the Hessian of $E_2$ at a critical point $\phi$ is given by
    \begin{equation*}
    H_{\phi}(V_1,V_2) = \int_M \langle I(V_1), V_2\rangle \dv,
\end{equation*}
where $I:\Gamma(\phi^*TN)\rightarrow \Gamma(\phi^*TN)$ is a linear elliptic self-adjoint differential operator.
\end{thm}

\pagebreak

Since $M$ is compact, the spectrum 
$\lambda_1<\lambda_2<\dots$
of $I$ is discrete 
and $\lambda_i$ tends to $\infty$ as $i$ goes to $\infty$.
In particular, the eigenspace $E_i$ associated with $\lambda_i$ is finite for each $i\in\mathbb{N}$. The index of $\phi$ is defined by
\begin{align*}
 \mathrm{Index}(\phi)=\sum_{\lambda_i<0}\dim(E_i).   
\end{align*}

\smallskip

Recall that we write $T_{\vert g \cdot \dot{\gamma}(t)}=g \cdot \dot{\gamma} (t)$ for the unit vector field normal to the principal orbits. If $f$ is a smooth function on the biharmonic orbit $G \cdot \gamma(t)$, the second variation of the bienergy for any variation vector field of the form $V=f T$ reads
\begin{equation}\label{2ndvariation}
    \begin{split}
        H_{\phi}(V,V) =& \int_M ( \Delta f)^2 +  4|\operatorname{S}_t \nabla f|^2  \\
        &-\tfrac{1}{2} f^2 \operatorname{trace}(P_t^{-1} \dot{P}_t) \left[ \tfrac{d}{dt} \operatorname{Ric}(e_0, e_0) - 4\operatorname{Tr}_g R(e_0, \operatorname{S}_t \cdot, \cdot, e_0) \right]   \dv_g.
    \end{split}
\end{equation}

The previous equation is a consequence of~\cite[Theorem~2.1]{OuStability} after noting that $\phi_0$ is biharmonic and $\operatorname{Ric}(e_0)^{T}=0=\operatorname{Tr}_g R(e_0, \cdot, \nabla_{e_0} \cdot, e_0)$ since $g$ is a diagonal cohomogeneity one metric (see~\cite[Proposition~1.12]{GroveZiller}).

On the one hand, we have that
$$
-4\operatorname{Tr}_g R(e_0, \operatorname{S}_t \cdot, \cdot, e_0) = \tfrac{1}{2} \operatorname{trace} (P_t^{-1} \dot{P}_t)^3 - \operatorname{trace}P_t^{-1} \dot{P}_t P_t^{-1} \ddot{P}_t.
$$
Moreover,
$$
\tfrac{d}{dt} \operatorname{Ric}(e_0, e_0) = \tfrac{1}{4} \tfrac{d}{dt} \operatorname{trace} \left[ (P_t^{-1}\dot{P}_t)^2 - 2P_t^{-1} \ddot{P}_t \right].
$$

Since $\operatorname{trace}$ is a linear operator, we obtain
$$
\tfrac{d}{dt} \operatorname{Ric}(e_0, e_0) = - \tfrac{1}{2} \operatorname{trace}(P_t^{-1}\dot{P}_t)^3 + \operatorname{trace} P_t^{-1}\dot{P}_t P_t^{-1} \ddot{P}_t - \tfrac{1}{2} \operatorname{trace} P_t^{-1} \dddot{P}_t.
$$

The previous identities yield the following result.
\begin{proposition}\label{coh1stab}
    Let $\phi_t: G \cdot \gamma(t) \to M$ be a biharmonic isometric immersion into a cohomogeneity one manifold with normal geodesic $\gamma(t)$ and diagonal metric. The second variation of the bienergy for any normal variation $f T$ is given by
    $$
    H_{\phi}(fT,fT) = \int_M ( \Delta f)^2 +  4|\operatorname{S}_t \nabla f|^2 + \tfrac{1}{4}  \operatorname{trace}(P_t^{-1} \dot{P}_t)  \operatorname{trace} (P_t^{-1} \dddot{P}_t) f^2 \dv_g.
    $$
\end{proposition}

As a direct application, we get the following algebraic criteria to check if the solution is unstable.
\begin{corollary}\label{corollarystability}
    Let $M$ be a homogeneous biharmonic hypersurface of a cohomogeneity one manifold with diagonal metric. If
    $$
    \operatorname{trace}(P_t^{-1} \dot{P}_t)  \operatorname{trace} (P_t^{-1} \dddot{P}_t) <0
    $$
    then it is unstable.
\end{corollary}

\begin{example}
    In Example~\ref{quadric} we showed that the orbit $\operatorname{SO}(n+1) \cdot \gamma(\frac{\pi}{4})$ is a biharmonic hypersurface in the complex quadric $Q_n$, $n >2$. In this case,
    $$
    \left( \operatorname{trace}(P_t^{-1} \dot{P}_t)  \operatorname{trace} (P_t^{-1} \dddot{P}_t) \right)_{\vert t= \frac{\pi}{4}} = -16 (n-2)^2 <0.
    $$
    By Corollary~\ref{corollarystability}, the biharmonic hypersurface is unstable.
\end{example}

\begin{example}
    For the biharmonic solutions~\ref{sphere1} and~\ref{sphere2}, since $|\operatorname{S}_t \nabla f |^2 = |\nabla f|^2$, we have
    $$
    H_{\phi}(V,V)=\int_{M_t} ((\Delta f)^2 - 4f \Delta f -4m^2f^2) \dv_g.
    $$
    It has been showed in~\cite{OuStability} that an analysis of the spectrum of $\mathbb{S}^n(\frac{1}{\sqrt{2}})$ and $\mathbb{S}^k(\frac{1}{\sqrt{2}}) \times \mathbb{S}^{n-k}(\frac{1}{\sqrt{2}})$ reveals that the index for normal variations of these biharmonic hypersurfaces in the sphere is $1$. In~\cite{LoubeauIndex}, the authors proved that the index (in general) for the biharmonic map $\mathbb{S}^{n}(\frac{1}{\sqrt{2}}) \to \mathbb{S}^{n+1}$ is also one, i.e., only the direction $T$ adds up to the index.
\end{example}

\begin{example}
    We cannot apply the same argument for the solutions of Example~\ref{CPn} since we do not know the explicit eigenvalues of the Laplace-Beltrami operator of a tube over a totally geodesic $\mathbb{CP}^{n-p}$ in $\mathbb{CP}^n$. However, if we denote by $\mu_1$ the first non-zero eigenvalue of the biharmonic tube, using the bound $\mu_1 \geq (n+1) - \frac{1}{4}|\operatorname{trace}P_t^{-1}\dot{P}_t|$, the first author concluded that for $n-p$ large enough, the biharmonic hypersurface obtained has normal index $1$, see~\cite{Balado2}.
\end{example}

We aim to exploit this technique to other examples. We recall the following result:
\begin{thm}\cite{Ho}\label{firsteigenvalue}
    Suppose that $ M $ is a compact orientable hypersurface embedded in a compact orientable Riemannian manifold $ N $. If the Ricci curvature of $ N $ is bounded below by a positive constant $ k $, then $ 2 \mu_1 > k - (n-1)\max_M | H| $ where $ \mu_1 $ is the first non-zero eigenvalue of the Laplace-Beltrami operator on $ M $ and $H$ denotes the mean curvature of $M$.
\end{thm}

\begin{example}
In this example, we study the stability of the biharmonic hypersurface \\${\operatorname{Sp}(n) \times \operatorname{Sp}(1) \cdot \gamma(t)}$ in $\mathbb{HP}^n$ from example~\ref{quaternionicprojectivespace}. Here, $t=\arccos \sqrt{x_-}$, where
    $$
    x_-=\frac{n+5-\sqrt{n^2+4n+13}}{4(n+2)}
    $$
    is one of the solutions from equation~\eqref{quaternioniceq}. 

    As $n \to \infty$, we have the following asymptotic relations
    \begin{equation*}
        \begin{split}
            \operatorname{trace}(P_t^{-1} \dot{P}_t)&=-12 \sqrt{3} n^{-\frac{1}{2}} + o(n^{-\frac{1}{2}}), \\
            \operatorname{trace}(P_t^{-1} \dddot{P}_t)&=48 \sqrt{3} n^{\frac{1}{2}} + o(n^{\frac{1}{2}}).
        \end{split}
    \end{equation*}
    Hence, by Corollary~\ref{corollarystability}, we can ensure that the biharmonic hypersurface is unstable if $n$ is large enough.
    
    In what follows we show that the normal index is not greater than $1$. Without loss of generality, we assume that $f$ is an eigenfunction of the Laplace-Beltrami operator on $\operatorname{Sp}(n) \times \operatorname{Sp}(1) \cdot \gamma(t)$ with eigenvalue $\mu$. Hence,
    $$
    H_{\phi}(fT,fT) \geq \left( \mu^2 + \tfrac{1}{4}  \operatorname{trace}(P_t^{-1} \dot{P}_t)  \operatorname{trace} (P_t^{-1} \dddot{P}_t)  \right) \langle f, f \rangle_{L^2}.
    $$
    Since $(\mathbb{HP}^n, g_{\text{can}})$ is Einstein with $\operatorname{Ric}=4(n+2) g_{\text{can}}$, we can use Theorem~\ref{firsteigenvalue} to obtain
    $$
    \mu_1 > 2(n+2) + \tfrac{1}{4} \operatorname{trace}(P_t^{-1} \dot{P}_t)=2n +o(n)
    $$
    as $n \to \infty$.
   Therefore,
   $$
   H_{\phi}(fT, fT) \geq (4n^2 + o(n^2)) \langle f, f \rangle_{L^2},
   $$
   and we get the existence of a constant $C>0$ such that if $n \geq C$, then the biharmonic hypersurface ${\operatorname{Sp}(n) \times \operatorname{Sp}(1) \cdot \gamma(t)}$ in $\mathbb{HP}^n$ has normal index $1$.
\end{example}

\section{\texorpdfstring{$r$}{r}-harmonic submanifolds}\label{rharm}

We can use the techniques developed in previous sections to derive the $r$-harmonic equation of a principal orbit in a cohomogeneity one manifold in terms of $P_t$.

\smallskip

In order to simplify the notation, we write $\alpha=-\frac{1}{2}\operatorname{trace}P_t^{-1}\dot{P}_t, \beta = \frac{1}{4} \operatorname{trace} (P_t^{-1} \dot{P}_t)^2$. Note that $\alpha$ is nothing but the trace of the shape operator, and $\beta$ represents the squared Frobenius norm of the shape operator.

\begin{thm}\label{polyharmonic}
    The orbit $G \cdot \gamma(t)$ of a cohomogeneity one action of $G$ in $M$ with a diagonal metric is $r$-harmonic if and only if
    \begin{equation*}
        \beta \operatorname{trace}P_t^{-1}\ddot{P}_t = \tfrac{(2-r)\alpha}{4} \operatorname{trace} (P_t^{-1} \dot{P}_t)((P_t^{-1} \dot{P}_t)^2 - 2 P_t^{-1} \ddot{P}_t).
    \end{equation*}
\end{thm}
\pagebreak
\begin{proof}
    In the proof of Theorem~\ref{hypersurface}, we have seen that $\Delta \tau = \beta \tau$. Hence, for any $k \in \mathbb{N}$:
    $$
    \Delta^{k} \tau = \beta^k \tau.
    $$
    In particular,
    $$
    \operatorname{Tr}_g R(\Delta^k \tau, \cdot)\cdot = \beta^k (\beta - \tfrac{1}{2} \operatorname{trace} P_t^{-1}\ddot{P}_t) \tau.
    $$
    Moreover, $\nabla_{E_i^*} \tau = \alpha \nabla_{E_i^*} e_0 = \tfrac{1}{2} \alpha (P_t^{-1}\dot{P}_t E_i)^*$, thus
    $$
    \operatorname{Tr}_g R(\nabla_{\cdot} \Delta^k \tau , \Delta^{\ell} \tau  ) \cdot = - \tfrac{\alpha}{8} \beta^{k+ \ell}  \operatorname{trace} (P_t^{-1} \dot{P}_t)((P_t^{-1} \dot{P}_t)^2-2P_t^{-1}\ddot{P}_t) \, \tau.
    $$
    Equations~\eqref{maeta-1} and~\eqref{maeta-2} yield the result.
\end{proof}

Even though the equation of the previous theorem may seem too difficult to investigate in general, there are examples for which the equation is easy to study.

\begin{example}
    If we consider again $(\mathbb{S}^{n+1},g_{\text{can}})$ with the natural action of $\operatorname{SO}(n+1)$, we have seen that we can take $P_t=\sin^2 t \, \mathds{1}_n$. In this situation, the $r$-harmonic equation reads
    $$
    r\sin^2 t =1,
    $$
    and we recover the $r$-harmonic immersions $\mathbb{S}^n (\tfrac{1}{\sqrt{r}}) \to \mathbb{S}^{n+1}$ from~\cite{MontaldoNewExamples}.
\end{example}

From this example, we obtain a tool to generate new $r$-harmonic submanifolds in the sphere. In the spirit of the composition property for biharmonic maps, see~\cite{LoubeauIndex}, we get the following result.

\begin{thm}
    Let $M$ be a compact manifold and $\psi: M \to \mathbb{S}^{m}(\frac{1}{\sqrt{r}})$ a harmonic map with constant energy density. Then the composition $f=\phi \circ \psi: M \to \mathbb{S}^{m}(\frac{1}{\sqrt{r}}) \to \mathbb{S}^{n+1}$ is proper $r$-harmonic.
\end{thm}
\begin{proof}
    The tension field of the composition for two smooth maps reads
    $$
    \tau(f)=d\phi (\tau(\psi)) + \operatorname{Tr}_g \nabla d\phi (d\psi, d\psi).
    $$
    Up to isometries on the codomain, we can always assume that for the immersion $\phi$ we have $n-m+1$ orthogonal unit normal vectors $\eta_k$ satisfying that the shape operators $\operatorname{S}_{\eta_k}=0$ for $k \geq 2$ and $\operatorname{S}_{\eta_1}=-\sqrt{r-1} \, \mathds{1}_n$. Hence, since $\psi$ is harmonic we have that
    $$
    \tau(f)=- 2\sqrt{r-1} e(\psi) \, T.
    $$
    Here $e(\psi)=\frac{1}{2} |d\psi|^2$ denotes the energy density of the map $\psi$. 

    A computation shows
    $$
    \Delta \tau(f) = -2\sqrt{r-1} e(\psi) \Delta T = 2 (r-1) e(\psi) \tau(f),
    $$
    and since $e(\psi)$ is constant on $M$, we get
    $$
    \Delta^k \tau(f) = 2^k (r-1)^k e(\psi)^k \tau(f).
    $$
    On the other hand,
    $$
    \operatorname{Tr}_g R(\tau(f), df) df = 2e(\psi)\tau(f).
    $$
    Together with equations~\eqref{maeta-1} and~\eqref{maeta-2} and the fact that $\nabla_{X} \tau(f) = -2e(\psi)(r-1) d \psi (X)$ for every $X \in T_p M$, we obtain $\tau_r(f)=0$.
\end{proof}

\begin{example}
    Lawson~\cite{lawson1970complete} proved the existence of an embedded minimal surface in $\mathbb{S}^3$ of arbitrary genus. Hence, from the previous theorem, we obtain examples of proper $r$-harmonic surfaces of arbitrary genus in $\mathbb{S}^{n+1}$.
\end{example}

\begin{example}
    If we apply Theorem~\ref{polyharmonic} to the $\operatorname{SO}(k+1) \times \operatorname{SO}(n-k+1)$-action on $(\mathbb{S}^{n+1},g_{\text{can}})$ we recover every biharmonic isoparametric hypersurface in the sphere, see~\cite[Theorem~1.2]{MontaldoNewExamples} and~\cite{PolyharmonicSpaceForms}.
\end{example}

\pagebreak

\begin{example}
    Regarding $\mathbb{CP}^n$, Theorem~\ref{polyharmonic} yields every $\operatorname{SU}(p+1) \times \operatorname{SU}(n-p)$-homogeneous $r$-harmonic hypersurface. For a detailed discussion on the classification of homogeneous $r$-harmonic hypersurfaces in $\mathbb{CP}^n$, we refer to~\cite{Balado2}.
\end{example}

\begin{example}
    For a surface of revolution $M^2$ generated by the profile curve $\gamma(t)=(\phi(t),0,\psi(t))$ parametrized with unit speed, the intersection of $M$ with a plane normal to the axis of revolution at $\gamma(t_0)$ is $r$-harmonic if and only if the following equation holds:
    $$
    (\phi'(t_0))^2=(1-r)\phi(t_0) \phi''(t_0).
    $$
\end{example}

\begin{example}
    For the cohomogeneity one action of $\operatorname{SO}(n+1)$ on the complex quadric $Q_n= \operatorname{SO}(n+2)/\operatorname{SO}(2)\times \operatorname{SO}(n)$, we have seen in Example~\ref{quadric} that we can write for every $t \in (0, \frac{\pi}{2})$:
    \begin{equation*}
        P_t = \begin{pmatrix}
            \cos^2 t  & {} & {} \\
            {} & \mathds{ 1 }_{ n-1 } & {} \\
            {} &  {} &  \sin^2  t \, \mathds{ 1 }_{ n-1 }
        \end{pmatrix}.
    \end{equation*}

    However, from Theorem~\ref{polyharmonic} we see the equation we obtain for $P_t$ is exactly the same for $\tilde{P}_t$, where
    \begin{equation*}
        \tilde{P}_t = \begin{pmatrix}
            \cos^2 t  & {} \\
              {} &  \sin^2  t \, \mathds{ 1 }_{ n-1 }
        \end{pmatrix}.
    \end{equation*}
    These are the endomorphisms for the case of the $\operatorname{SO}(2) \times \operatorname{SO}(n)$-action on $(\mathbb{S}^{n+1},g_{\text{can}})$.  Hence, an orbit is proper $r$-harmonic if and only if $x=\cos^2 t$ is a root of
    $$
    P(x)=rnx^3 + (n-2-r(n+1))x^2+(r+2)x-1,
    $$
    from where we obtain $r$-harmonic orbits for every $r$.
\end{example}

\begin{example}
    Recall that for the quaternionic projective space $(\mathbb{HP}^n, g_{\text{can}})$ with the action of $\operatorname{Sp}(n) \times \operatorname{Sp}(1)$ we have
    \begin{equation*}
        P_t = \begin{pmatrix}
            \sin^2 t \, \mathds{ 1 }_{ 4(n-1) } & {} \\
            {} & \sin^2 2t \, \mathds{ 1 }_{ 3 }
        \end{pmatrix}.
    \end{equation*}
    for every $t \in (0, \frac{\pi}{2})$. A lengthy computation shows that the $r$-harmonic equation is in this case equivalent to finding the roots of the polynomial
    $$
    P(x)=16a_4x^4+8a_3x^3+24a_2x^2-24a_1x+18
    $$
    where $x=\cos^2 t$ and
    \begin{equation*}
        \begin{split}
            a_4&=(2n^2+11n+5)r-6(n-1), \\
            a_3&=-(4n^2+37n+31)r+2(2n+13)(n-1), \\
            a_2&=5(n+2)r-3(n-2), \\
            a_1&=3r+n+2.
        \end{split}
    \end{equation*}
    The existence of proper $r$-harmonic solutions follows from the fact that $P(0)>0$ and \\${P(\frac{3}{2(2n+1)})<0}$ for every $r$.
\end{example}

\begin{remark}
    In the case of $\mathbb{S}^2 \times \mathbb{S}^2$ with the cohomogeneity one diagonal metric as in Example~\ref{S2xS2}, a numerical analysis suggests that there are no proper $r$-harmonic orbits for any $r \geq 2$.
\end{remark}

\pagebreak

\section{Polyharmonic foliations}\label{SecFoliations}

Manifolds carrying a foliation $\mathcal{F}$ where every leaf is minimal have been extensively studied. We aim in this section to construct foliations where every leaf is a biharmonic (respectively $r$-harmonic) submanifold.

\subsection{Warped products}
Caddeo, Montaldo, and Piu determined in~\cite{Caddeocurves} the surfaces of revolution all of whose parallels are biharmonic curves. In this spirit, we start by considering now biharmonic foliations in warped product manifolds.

\smallskip

Consider the product manifold $M=[0,\infty)\times \mathbb{S}^n$ equipped with a warped product metric of the form
$$
g=dt^2 + f(t)^2 g_{\mathbb{S}^n}.
$$
$M$ admits a natural $\operatorname{SO}(n+1)$-cohomogeneity one action. Note that in this case, the manifold $M$ is not necessarily compact, however, the discussion of Subsection~\ref{SectionHypersurfaces} follows similarly for this case. Using the same notation, we can write for $t \in (0,\infty)$
$$
P_t=f(t)^2 \, \mathds{ 1 }_{ n }.
$$

The condition for $r$-harmonicity from Theorem~\ref{polyharmonic} is hence given by
$$
(f')^2+(r-1)ff''=0.
$$
If we consider the previous equation as a second-order ordinary differential equation on $f$, the family of functions
$$
f(t)=c_1 (c_2 + t)^{\frac{r-1}{r}}, \qquad c_1,c_2>0,
$$
are solutions to the aforementioned equation.

Take the partition
$$
\mathcal{F}=\{t \times \mathbb{S}^n(f(t))\}_{t \in [0,\infty)}.
$$
In particular, $(M, dt^2+t^{2-\frac{2}{r}}g_{\mathbb{S}^n}, \mathcal{F})$ is a singular Riemannian foliation where each leaf is proper $r$-harmonic for any $t>0$.

For the case $r=2$, it is possible to study the normal stability for each of the biharmonic maps $\phi_t: \mathbb{S}^n(\sqrt{t}) \to M$. Note that $\operatorname{trace}P_t^{-1}\dddot{P}_t=0$ and that the shape operator of $\phi_t(\mathbb{S}^n (\sqrt{t}))$ is given by $\operatorname{S}_t=-\frac{1}{2t} \mathds{ 1 }_{ n }$. Without loss of generality, we consider normal variations of the form $fT$ where $f$ is an eigenfunction of the Laplace-Beltrami operator on $\mathbb{S}^n(\sqrt{t})$ with eigenvalue $\mu$.  Hence, by Proposition~\ref{coh1stab} we have
$$
    H_{\phi}(fT,fT)= \mu [\mu-t^{-2}] \, \langle f, f \rangle_{L^2}.  
$$
The spectrum of the Laplace-Beltrami operator in $\mathbb{S}^n(\sqrt{t})$ is given by $\{t^{-1}k(n+k-1): k \in \mathbb{N}\}$. Thus, the sign of $H_{\phi}(fT,fT)$ is completely determined by $k(n+k-1)-t^{-1}$.

We conclude that the normal index of $\phi_t$ is the number of eigenvalues (with multiplicity) smaller than $t^{-1}$. Explicitly,
$$
\sum_{j=1}^k \frac{(n+2j-1)(n+j-2)!}{j!(n-1)!} , \quad \text{where} \quad k=\max_{\ell \in \mathbb{N}} \{t\ell (n+\ell-1)<1\}.
$$
The normal nullity of $\phi_t$ is given by the number of eigenvalues (with multiplicity) equal to $t^{-1}$ plus $1$ (corresponding to the zero eigenvalue). Explicitly,
$$
\frac{(n+2k-1)(n+k-2)!}{k!(n-1)!)}+1 \quad \text{if} \quad tk (n+k-1)=1
$$
and $1$ otherwise.

\subsection{Doubly warped products}
One can use a similar strategy to obtain more intricate constructions. 

\smallskip

Consider for example the manifold $M=I \times \mathbb{S}^n \times \mathbb{S}^m$, where $I$ is an open interval to be determined, endowed with the doubly warped product metric
$$
g= dt^2 + f(t)^2 g_{\mathbb{S}^n} + h(t)^2 g_{\mathbb{S}^m}.
$$
The manifold $M$ admits a cohomogeneity one action of $\operatorname{SO}(n+1) \times \operatorname{SO}(m+1)$. Using Theorem~\ref{hypersurface}, the biharmonic equation for the immersions $\mathbb{S}^n(f(t)) \times \mathbb{S}^m(h(t)) \to M$ reads
$$
n\left( \tfrac{f'}{f}\right)^2 + n\tfrac{f''}{f} + m\left( \tfrac{h'}{h}\right)^2 + m\tfrac{h''}{h}=0.
$$

Hence, writing $I=(-\frac{\pi}{4} \sqrt{\frac{m}{n}}, \frac{\pi}{4} \sqrt{\frac{m}{n}})$, each leaf of the foliation $\mathcal{F}=\{t \times\mathbb{S}^n(f(t)) \times \mathbb{S}^m(h(t))\}_{t \in I}$ of $M$ equipped with the metric
$$
g=dt^2 + e^{2t} g_{\mathbb{S}^n} + \cos(2t \sqrt{\tfrac{n}{m}}) g_{\mathbb{S}^m},
$$
is biharmonic. Note that if $t = \frac{1}{2}\sqrt{\frac{m}{n}}\arctan \sqrt{\frac{n}{m}}$ the hypersurface is minimal, in any other case it is proper biharmonic.

\subsection{Foliations on closed manifolds}
Finding biharmonic foliations on closed manifolds seems to be a challenging task. In this subsection we construct a $\mathcal{C}^1$-foliation $\mathcal{F}$ of the flat torus $\mathbb{T}^2=\mathbb{R}^2/\Lambda$, $\Lambda = \{(n,m) : n,m \in \mathbb{Z}\}$, satisfying the following property: there exists a geodesic $\sigma$ of $\mathbb{T}^2$ such that $\mathcal{F}$ restricts to a $\mathcal{C}^{\infty}$-foliation of $\mathbb{T}^2 \backslash \sigma$ whose leaves are proper biharmonic curves. A similar strategy can be carried out for the $r$-harmonic case with $r>2$.

\smallskip

Consider the family of functions
$$
\psi_{a,b,c,d}: \mathbb{R} \to \mathbb{R}, \qquad t \mapsto at^3+bt^2+ct+d,
$$
where $a,b,c,d \in \mathbb{R}$. These maps are proper biharmonic whenever $a\neq 0$ or $b\neq 0$. By a result of Ou~\cite{ou2012some}, the graph of a proper biharmonic function is a proper biharmonic immersion in the product space. Therefore, the graphs $\Gamma(\psi_{a,b,c,d})=\{(t,\psi_{a,b,c,d}(t)) : t \in \mathbb{R}\}$ are proper biharmonic curves in $\mathbb{R}^2$.

To obtain closed curves in the torus, we modify the functions $\psi_{a,b,c,d}$ to be $1$-periodic. For this, we impose the condition $\psi_{a,b,c,d}(0)=\psi_{a,b,c,d}(1)$, from which we obtain the equation
\begin{equation}\label{condition1}
    a+b+c=0,
\end{equation}
and define ${\Psi}_{a,b,c,d}(t)=\psi_{a,b,c,d}(t-\lfloor t \rfloor)$ for $t \in \mathbb{R}$ and $a,b,c,d$ satisfying~\eqref{condition1}. The graphs $\Gamma(\Psi_{a,b,c,d})$ descend to a family of closed $\mathcal{C}^0$-curves in $\mathbb{R}^2/\Lambda$.

If we impose further $\psi_{a,b,c,d}'(0)=\psi_{a,b,c,d}'(1)$, we obtain
\begin{equation}\label{condition2}
    3a +2b=0.
\end{equation}
Combining~\eqref{condition1} and~\eqref{condition2} we obtain a family of $\mathcal{C}^1$-curves $\Gamma(\Psi_{2a,-3a,a,d})$ which are $\mathcal{C}^{\infty}$ in $\mathbb{R}^2$ outside of $\{(n,y): n\in \mathbb{Z}, y \in \mathbb{R}\}$.

If we asked for more regularity, then the condition $\psi_{2a,-3a,a,d}'(0)=\psi_{2a,-3a,a,d}'(1)$ 
would only yield straight lines.

The $C^1$-foliation $\tilde{\mathcal{F}}=\{\Gamma(\Psi_{2,-3,1,d})\}_{d \in \mathbb{R}}$ of $\mathbb{R}^2$ descends to a $\mathcal{C}^1$-foliation $\mathcal{F}$ of $\mathbb{T}^2$ formed by closed curves. The leaves are $\mathcal{C}^{\infty}$ and proper biharmonic outside of the geodesic arising from the set $\{(n,y): n\in \mathbb{Z}, y \in \mathbb{R}\}$.

\pagebreak

\bibliographystyle{plain}
\bibliography{Bibliography.bib}

\end{document}